\documentclass[12pt]{article}
\usepackage[dvips]{graphicx}
\usepackage{amsmath,amsfonts,amssymb,amsthm,color}
\usepackage{boxedminipage}

\usepackage[numbers]{natbib}

\newtheorem{theorem}{Theorem}
\newtheorem{lemma}{Lemma}

\newtheorem{proposition}{Proposition}

\usepackage{bm}
\bmdefine{\Bt}{t}
\bmdefine{\BX}{X}
\bmdefine{\BY}{Y}
\bmdefine{\BZ}{Z}
\bmdefine{\BB}{B}
\bmdefine{\BM}{M}
\bmdefine{\BD}{D}
\bmdefine{\Bi}{i}
\bmdefine{\Bj}{j}
\bmdefine{\Bk}{k}
\bmdefine{\Bx}{x}
\bmdefine{\By}{y}
\bmdefine{\Bz}{z}
\bmdefine{\Bv}{v}
\bmdefine{\Bw}{w}
\bmdefine{\Bn}{n}
\bmdefine{\Ba}{a}
\bmdefine{\Bb}{b}
\bmdefine{\Bc}{c}
\bmdefine{\Be}{e}
\bmdefine{\Bu}{u}
\bmdefine{\Bp}{p}
\bmdefine{\Bzero}{0}
\bmdefine{\Bone}{1}

\newcommand{\C}{{\mathbb C}}

\newcommand{\interm}{{\operatorname{in}_{\succ}}}

\def\diag{{\mathop {\rm diag}}\nolimits}

\def\2{{1 \over \,2\,}}

\def\diag{\mathop{{\rm diag}}}
\def\tr{\mathop{\rm tr}}
\def\hyperF#1#2{{}_{#1}\kern-.05emF_{#2}}

\def\pdop{\partial}

\newcommand{\partitionof}{\vdash}
\newcommand{\der}{\partial}
\newcommand{\initial}{\operatorname{in}}
\newcommand{\ch}{\operatorname{ch}}
\newcommand{\CC}{\mathbb{C}}

\newcommand{\dominatedby}{\unlhd}

\newcommand{\dominates}{\unrhd}
\newcommand{\concat}{\uplus}

\newcommand{\emptypartition}{\varnothing}

\title{Holonomic gradient method for the distribution function
of the largest root of a Wishart matrix}
\author{
Hiroki Hashiguchi\thanks{
Graduate School of Science and Engineering, Saitama University} , 
Yasuhide Numata\thanks{Department of 
Mathematical Informatics, 
Graduate School of Information Science and Technology, 
University of Tokyo} \thanks{JST CREST} , 
Nobuki Takayama\thanks{Department of Mathematics, Kobe University}
\footnotemark[3]\\
and Akimichi Takemura\footnotemark[2] \footnotemark[3]
}
\date{January 2012}

\begin{document}

\maketitle

\begin{abstract}
We apply the holonomic gradient method introduced by \citet{hgd} 
to the evaluation of the exact distribution function of the
largest root of a Wishart matrix, which involves a hypergeometric
function $\hyperF{1}{1}$ of a matrix argument.  Numerical evaluation of the
hypergeometric function has been one of the longstanding problems in
multivariate distribution theory. The holonomic gradient method offers
a totally new approach, which is complementary to the infinite series
expansion around the origin in terms of zonal polynomials.  It allows
us to move away from the origin by the use of partial differential
equations satisfied by the hypergeometric function.  From numerical
viewpoint we show that the method works well up to dimension 10.  From
theoretical viewpoint the method offers many challenging problems both
to statistics and $D$-module theory.
\end{abstract}


\section{Introduction}
For multivariate distribution theory in statistics, the
theory of zonal polynomials and hypergeometric functions of 
matrix arguments, introduced by A.T.\ James and other authors,
was a very important development in the 1950's. They allowed
explicit expressions of density functions and 
cumulative distribution functions of basic test statistics
under non-null cases.  Zonal polynomials are based
on the representation theory of real general linear group and
they possess many interesting combinatorial properties.
Properties and applications of zonal
polynomials and hypergeometric functions of  matrix arguments
are surveyed in \citet{gross-richards-1987} and \citet{richards-2010}.
Zonal polynomials are special cases of Jack polynomials, whose
properties have been intensively studied by many mathematicians.
See for example Chapters VI and VII of \citet{macodonald-book}
and \citet{stanley-1989am}. 
Jack polynomials are further generalized to Macdonald polynomials
(see, e.g., \citet{kuznetsov-sahi-2006}).  

Zonal polynomials and hypergeometric functions of matrix arguments 
are important and difficult to 
compute in non-null cases rather than the null case, where the covariance matrix is a  
multiple of the identity matrix.
In the null case there are several approaches to obtain the distribution function or moments.
Recent representative approach is to use the random matrix theory (RMT) and 
the landmark study on the connection  between RMT and multivariate analysis was conducted by 
\citet{johnstone-2001, johnstone-2008}.
\citet{bulter-paige-2011} proposed a method to compute the exact null distributions based on 
their Pfaffian representation given by  \citet{gupta-richards-1985}. 

Despite the above nice mathematical properties of zonal polynomials and
hypergeometric functions of matrix arguments, from practical viewpoint
they were not really useful for computations.  Coefficients of zonal polynomials
can be computed only through nontrivial combinatorial recursions.
Although very ingenious recursion algorithms have been recently developed
(\citet{koev-edelman}), computing zonal polynomials of large degrees
remains to be a difficult problem  because of inherent combinatorial complexities.
Also, the convergence of infinite series expansion of  
hypergeometric functions of a  matrix argument in terms of zonal
polynomials was found to be slow (\citet{muirhead-1978}, \citet{hashiguchi-niki-2006}).
Since the expansion of the hypergeometric function in terms of zonal polynomials
is the expansion {\it at the origin}, the convergence for large values of the argument
is necessarily slow.

The holonomic gradient method allows us to move away from the origin by the
use of partial differential equations.
Thus our approach provides a promising new method for attacking 
a longstanding problem in multivariate statistics.
Our holonomic gradient method is, in spirit, on the track of the holonomic
systems approach to combinatorial identities by 
\citet{zeilberger-1990jcam}.
Note that the series expansion and our holonomic
gradient method are in fact complementary methods, because our method needs the 
series expansion for obtaining initial values
for the partial differential equations. 

The main purpose of this paper is to verify the performance of holonomic
gradient method for $\hyperF{1}{1}$.  We found that a straightforward implementation of the
holonomic gradient method works well for dimensions up to 10.

\citet{butler-wood-2002} showed that
the Laplace method gives a very good approximation to $\hyperF{1}{1}$ even for
a high dimension, e.g., $m=32$. However the Laplace method needs a peaked density function, which
corresponds to a large degrees of freedom.  Our method is an exact method, where the 
errors only come from discretization in numerically solving differential equations and
the accuracies in the initial values.
Hence our method works even for small degrees of freedom.

The organization of this paper is as follows.
In Section \ref{sec:preliminaries} we summarize preliminary facts on the
exact distribution of the largest root of a Wishart matrix.  In particular we state
the partial differential equation for $\hyperF{1}{1}$ by \citet{muirhead-1970}.
In Section \ref{sec:dim2}, for expository purposes, we fully describe our holonomic
gradient method for dimension two.
In Section \ref{sec:general-dimension} we derive properties of Pfaffian system for 
general dimensions.
The Pfaffian system is a system of partial differential equations
and is called an integrable connection in some literatures.
Results of symbolic computations are presented in Section \ref{sec:symbolic-computation} and
results of numerical experiments are presented in Section \ref{sec:numerical}. 
We end the paper with discussion of open problems in 
Section \ref{sec:discussion}.

\section{Preliminaries}
\label{sec:preliminaries}

Let $\kappa=(k_1,\dots, k_l) \partitionof k$ be 
a partition of a non-negative integer $k$ and define
the Pochhammer symbol $(a)_\kappa$ by 
\begin{eqnarray}
(a)_{\kappa} = \prod_{i=1}^{l}\left(a-\frac{i-1}{2}\right)_{k_{i}}, 
\quad (a)_{k_{i}} = \prod_{j=1}^{k_{i}}(a+j-1) \ \ \ 
((a)_{0} = 1).
\nonumber
\end{eqnarray}
Let ${{\cal C}_\kappa(Y)}$ denote the (``$C$-normalization'' of)  zonal
polynomial indexed by $\kappa$ 
of an $m\times m$ symmetric matrix $Y$. It
is a homogeneous symmetric
polynomial of degree $k$ in the characteristic roots $y_1,\dots, y_m$ of $Y$,
satisfying $\sum_{\kappa\partitionof k} {\cal C}_\kappa(Y)=(\tr Y)^k$. 
For zonal polynomials  in statistics 
see, e.g., \citet{james-1964ams},  \citet{muirhead-book}, \citet{takemura-zonal} and
\citet{mathai-provost-hayakawa}.
A hypergeometric function of a matrix argument is defined
(\citet{constantine-1963}) as
\begin{equation}
\label{eq:def-hyper}
\hyperF{p}{q}(a_1,\dots, a_p;c_1, \dots, c_q; Y)=
\sum_{k=0}^\infty \sum_{\kappa \partitionof k} 
\frac{(a_1)_\kappa \dots (a_p)_\kappa}{(c_1)_\kappa \dots (c_q)_\kappa}
\frac{{\cal C}_\kappa(Y)}{k!}.
\end{equation}

In this paper we study holonomic gradient method for $\hyperF{1}{1}(a;c;Y)$.
Let  $I_m$ denote the $m\times m$ identity matrix and
let $|X|$ denote the determinant of $X$. 
For $\Re a > (m+1)/2$,  $\Re (b-a)> (m+1)/2$, 
$\hyperF{1}{1}(a;c;Y)$ has the following integral representation
\begin{equation}
\label{eq:integral-representation}
\hyperF{1}{1}(a;c;Y)=\frac{\Gamma_m(b)}{\Gamma_m(a)\Gamma_m(c-a)}
\int_{0 < X < I_m} \exp(\tr XY) |X|^{a-(m+1)/2}|I_m - X|^{c-a-(m+1)/2} dX,
\end{equation}
where  $0 < X < I_m$ means that $X$ and $I_m-X$ are positive definite, 
$dX=\prod_{i\le j}dx_{ij}$ 
is the Lebesgue measure of the upper triangular entries of $X$, and
\[
\Gamma_{m}(a) = \pi^{\frac{1}{4}m(m-1)}
\prod_{i=1}^{m}\Gamma \left(a-\frac{i-1}{2}\right).
\]
The hypergeometric function
$\hyperF{1}{1}$ satisfies the the following Kummer relation (see (2.8) of \citet{herz-1955}, (51) of \citet{james-1964ams}):
\begin{equation}
\label{eq:kummer}
\exp(- \tr Y) \hyperF{1}{1}(a;c;Y)=\hyperF{1}{1}(c-a, c; -Y).
\end{equation}
Note that \eqref{eq:integral-representation} implies that $\hyperF{1}{1}$ is an entire function
in $Y$.

The cumulative distribution function of the
largest root $\ell_1$ of the  $m\times m$ Wishart matrix $W$ with $n$ degrees of freedom
and the covariance matrix $\Sigma$ is written
as follows 
\begin{equation}
\Pr[\ell_{1} < x] = C \exp \left(-\frac{x}{2} \tr \Sigma^{-1} \right)
x^{\frac{1}{2}nm} \hyperF{1}{1} \left(\frac{m+1}{2};\frac{n+m+1}{2};
\frac{x}{2} \Sigma^{-1}\right) ,
\label{eq:2-1}
\end{equation}
where 
\[
C = \frac{\Gamma_{m}
\left(\frac{m+1}{2}\right)}{2^{\frac{1}{2}nm}
(\det \Sigma)^{\frac{1}{2} n} \Gamma_{m}
\left(\frac{n+m+1}{2}\right)}.
\]
This follows from the results in Section 9 of \citet{constantine-1963}
and the Kummer relation \eqref{eq:kummer}.  See also \citet{sugiyama-1967ams}.


The following  partial differential equations for $\hyperF{1}{1}(a;b;Y)$ 
were derived by 
\citet{muirhead-1970}.
\begin{theorem}[Theorem 5.1 of \citet{muirhead-1970}, Theorem 7.5.6 of \citet{muirhead-book}]
\label{thm:muirhead}
The hypergeometric function  $F=\hyperF{1}{1}(a; c; Y)$ of a matrix argument $Y=\diag(y_1,\dots, y_m)$
is the unique solution of the following set of $m$ partial differential
equations
\begin{align} \label{eqn:PDE1F1}
\left[ y_i \, \pdop_i^2 + \left\{ c - \frac{m-1}{2} -y_i + \2 \sum_{j=1, j \neq i}^m
\frac{y_i}{y_i - y_j} \right \} \pdop_i - 
\2  \sum_{j=1, j \neq i}^m \dfrac{y_j}{y_i - y_j}  \pdop_j 
-  a \right]&F=0, \\
& (i=1, \dots, m),\nonumber
\end{align}
subject to the conditions  that  $F$ is symmetric in $y_1, \dots, y_m$ and 
$F$ is analytic at $Y=0$, $F(0)=1$.
\end{theorem}
The partial differential equation
\eqref{eqn:PDE1F1}  has  singularities along $y_i=0$ and 
$y_j=y_i$, $j\neq i$.  However since $F$ is an entire function,
$F$ is determined by the partial differential equations on the  open
region ${\cal X}=\{ y\in \mathbb{C}^m \mid \prod_{i=1}^m y_i \prod_{i\neq j}(y_i - y_j) \neq 0
\}$.  In this paper we call ${\cal X}$ the non-diagonal region.
Using
\[
\frac{y_i}{y_i - y_j} = 1 + \frac{y_j}{y_i-y_j}
\]
we can rewrite  \eqref{eqn:PDE1F1} as $g_i F=0$, $i=1,\dots,m$, where
\begin{equation}
\label{eqn:PDE1f1N}
g_i =  y_i \pdop_i^2 + (c-y_i) \pdop_i +  \2 \sum_{j=1, j \neq i}^m
\frac{y_j}{y_i - y_j} (\pdop_i - \pdop_j) - a 
\end{equation}
is a differential operator annihilating $F$.
In our holonomic gradient method we make a direct use of the 
partial differential equations for numerical evaluation of  $\hyperF{1}{1}$.

\section{Holonomic gradient method for dimension two}
\label{sec:dim2}
In this section we illustrate the holonomic gradient method for the
case of $m=2$.  Although our purpose is to implement an algorithm
of our method for a larger dimension, for clarity it is best 
to do ``by hand'' calculation for the case of $m=2$.
As in the previous section we simply write $F(Y)=\hyperF{1}{1}(a;c;Y)$.

In \citet{hgd} the holonomic gradient method was used to obtain
the maximum likelihood estimate.  The reciprocal of the likelihood
function was minimized and the method was called the holonomic gradient
{\em descent}.  For the application of this paper we simply use
the holonomic gradient method for evaluating $F$.  Hence
we omit the term ``descent''.  Also, for minimization, at each
step of the iteration, a direction for increments was chosen to
decrease the value of the function.  In our application, starting
from the origin $Y=0$, we can choose arbitrary path to the
target value $Y$ where we want to evaluate $F(Y)$.

Another minor difference of the expository explanation in this section
from \citet{hgd} and \citet{so3}  is that we use the simple forward 
Euler method (e.g., Section 3.1 of \citet{ascher-petzold}) for updating
partial derivatives of $F$.
In \citet{hgd}, once an updating direction is chosen
at each step of the iteration, the 4-th order Runge-Kutta method was
used.  The simple Euler method is used {\it only for the purpose of exposition}.
It is easier to explain the basic idea of the holonomic gradient method
with the simple Euler method.
In our actual implementation in Section \ref{sec:numerical} we use the Runge-Kutta method
for numerically solving the differential equation.

We will reduce our problem to a traditional problem of numerical analysis of an ordinary
differential equation (ODE).  For the reduction we utilize the notion of holonomic differential
equations and the gradients of their solutions.  It is why we call our method holonomic gradient method.

In the following we discuss the case of $y_1\neq y_2$ and $y_1 = y_2$ separately.

\subsection{Holonomic gradient method for non-diagonal region} 
\label{subsec:m2-general}
In this subsection we assume $y_1 \neq y_2$.
Two partial differential equations in
\eqref{eqn:PDE1f1N} are written as
\begin{align}
\Big[y_1 \partial_1^2 + (c-y_1) \partial_1 + \frac{1}{2} \frac{y_2}{y_1 - y_2}
(\partial_1 - \partial_2)-a\Big] F &= 0, 
\label{eq:m2i1}\\
\Big[y_2 \partial_2^2 + (c-y_2) \partial_2 + \frac{1}{2} \frac{y_1}{y_2 - y_1}
(\partial_2 - \partial_1)-a\Big] F &= 0.
\label{eq:m2i2}
\end{align}
Suppose that we want to evaluate a higher derivative $\partial_1^{n_1} \partial_2^{n_2} F=\partial_2^{n_2} \partial_1^{n_1}F$ of $F$.
Let $n_2 \ge 2$.  Then by \eqref{eq:m2i2}
\begin{equation}
\label{eq:n22}
\partial_1^{n_1} \partial_2^{n_2} F 
=\partial_1^{n_1} \partial_2^{n_2-2} \big(-\frac{c}{y_2} \partial_2 + \partial_2
-  \frac{1}{2} \frac{y_1}{y_2(y_2 - y_1)}
(\partial_2 - \partial_1) + \frac{a}{y_2}\big)F.
\end{equation}
Noting 
\[
\partial_2 \frac{1}{y_2}= - \frac{1}{y_2^2}, \quad
\partial_2 \frac{y_1}{y_2(y_2-y_1)}= -\frac{y_1 (2y_2 - y_1)}{y_2^2 (y_2 - y_1)^2},
\]
for $n_2 > 2$, the right-hand side of  \eqref{eq:n22} is further written as
\begin{align}
&\partial_1^{n_1} \partial_2^{n_2-3} \Big(\frac{c}{y_2^2} \partial_2
- \frac{c-y_2}{y_2} \partial_2^2 
+  \frac{1}{2} \frac{y_1(2y_2-y_1)}{y_2^2(y_2 - y_1)^2}(\partial_2 - \partial_1)
\nonumber \\
& \qquad \qquad \ 
-  \frac{1}{2} \frac{y_1}{y_2(y_2 - y_1)}(\partial_2^2 - \partial_1 \partial_2)
-  \frac{a}{y_2^2}
+ \frac{a}{y_2}\partial_2\Big)F.
\label{eq:n22a}
\end{align}
Although the result is somewhat complicated,  the important fact is
that the total degree of differentiation $n_1+n_2$ on the left-hand side 
of \eqref{eq:n22} is decreased by one to $n_1 + n_2 -1$ in \eqref{eq:n22a}.
As long as the degree of $\partial_1$ or $\partial_2$ is more than one,
then we can recursively apply 
\eqref{eq:m2i1} or \eqref{eq:m2i2} to decrease the total degree of
differentiation.  It follows that for each $n_1,n_2$, there
exist rational functions 
$h^{(n_1, n_2)}_{00}, h^{(n_1, n_2)}_{10},h^{(n_1, n_2)}_{01},
h^{(n_1, n_2)}_{11}$ in $(y_1, y_2)$ 
such that
\begin{equation}
\label{eq:n1n2base}
\partial_1^{n_1}\partial_2^{n_2}F = 
h^{(n_1, n_2)}_{00} F + h^{(n_1, n_2)}_{10}\partial_1 F + h^{(n_1, n_2)}_{01}
\partial_2 F  + h^{(n_1, n_2)}_{11} \partial_1 \partial_2 F .
\end{equation}
In this notation 
\eqref{eq:m2i1} is written as
\begin{align}
\partial_1^2 F &= \frac{a}{y_1} F 
-(\frac{c-y_1}{y_1} +\frac{1}{2}\frac{y_2}{y_1(y_1-y_2)}) \partial_1 F 
+ \frac{1}{2}\frac{y_2}{y_1(y_1-y_2)} \partial_2 F \nonumber \\
&= h^{(2,0)}_{00} F + h^{(2,0)}_{10} \partial_1 F + h^{(2,0)}_{01}\partial_2 F
\qquad (h^{(2,0)}_{11}\equiv 0).
\label{eq:m2i1-h}
\end{align}
For a general dimension, \eqref{eq:n1n2base} corresponds to the reduction by a Gr\"obner basis as
discussed in Section \ref{sec:general-dimension}.

For us the important case is $n_1=1, n_2=2$.
Since 
\[
\partial_1 \frac{y_1}{y_2(y_2-y_1)}= \partial_1 \big( \frac{1}{y_2 - y_1} 
- \frac{1}{y_2}\big)= \frac{1}{(y_2-y_1)^2}
\] 
we have
\begin{align*}
\partial_1 \partial_2^2 F &
= \partial_1 
\big(-\frac{c-y_2}{y_2} \partial_2 
-  \frac{1}{2} \frac{y_1}{y_2(y_2 - y_1)}
(\partial_2 - \partial_1) + \frac{a}{y_2}\big)F\\
&= \big(
-\frac{c-y_2}{y_2} \partial_1\partial_2 
- \frac{1}{2} \frac{1}{(y_2-y_1)^2}(\partial_2 - \partial_1)  
-  \frac{1}{2} \frac{y_1}{y_2(y_2 - y_1)} (\partial_1 \partial_2 - \partial_1^2) 
+ \frac{a}{y_2}\partial_1\big) F.
\end{align*}
There is a term $y_1\partial_1^2$ on the right-hand side, into which we 
further substitute \eqref{eq:m2i1}.  Then
\eqref{eq:n1n2base} for $\partial_1 \partial_2^2 F$ is written as
\begin{align}
\partial_1 \partial_2^2 F &=
\Big(
-\frac{c-y_2}{y_2} \partial_1\partial_2  
- \frac{1}{2} \frac{1}{(y_2-y_1)^2}(\partial_2 - \partial_1)  
-  \frac{1}{2} \frac{y_1}{y_2(y_2 - y_1)} \partial_1 \partial_2 
+ \frac{a}{y_2}\partial_1\nonumber \\
& \qquad - 
\frac{1}{2y_2(y_2 - y_1)}\big((c-y_1) \partial_1 + \frac{1}{2}\frac{y_2}{y_1 - y_2}
(\partial_1 - \partial_2) - a \big)
\Big) F \nonumber\\
&= \frac{a}{2y_2(y_2 - y_1)} F  + 
\Big( \frac{3}{4} \frac{1}{(y_2-y_1)^2} + \frac{a}{y_2} 
- \frac{c-y_1}{2y_2(y_2 - y_1)} \Big)  \partial_1 F \nonumber\\
& \quad 
- \frac{3}{4} \frac{1}{(y_2-y_1)^2}\partial_2 F
- \Big(\frac{c-y_2}{y_2} + \frac{1}{2}\frac{y_1}{y_2(y_2 - y_1)} \Big)  \partial_1 \partial_2 F \nonumber \\
&= h^{(1,2)}_{00}F + h^{(1,2)}_{10} \partial_1 F + h^{(1,2)}_{01}\partial_2 F 
+ h^{(1,2)}_{11} \partial_1 \partial_2 F.
\label{eq:12-deriv-F}
\end{align}
Since $F$ is a symmetric function in $y_1$ and $y_2$, $\partial_1^2 \partial_2 F$
is obtained by permuting $y_1$ and $y_2$.

Let 
\[
\vec F = \begin{pmatrix}F \\ \partial_1 F\\ \partial_2 F\\ \partial_1 \partial_2 F
\end{pmatrix}
\]
denote the vector consisting of $F$ and its square-free mixed derivatives. 
Differentiate the components of $\vec F$ by $y_1$ and denote  $\partial_1 \vec F=(\partial_1 F, \partial_1^2 F, \partial_1 \partial_2 F, \partial_1^2\partial_2 F)^t$. Similarly define $\partial_2 \vec F$. Then 
by \eqref{eq:m2i1-h} and  \eqref{eq:12-deriv-F}, $\partial_i \vec F$, $i=1,2$, are written as
$\partial_i \vec F=P_i(Y)\vec F$, 
where  $P_1$ and $P_2$ are the following $4\times 4$ matrices with
rational function entries
\[
P_1(Y) = \begin{pmatrix} 0 & 1 & 0 & 0 \\
h^{(2,0)}_{00} & h^{(2,0)}_{10} & h^{(2,0)}_{01} & 0\\
0 & 0 & 0 & 1\\
h^{(2,1)}_{00} & h^{(2,1)}_{10} & h^{(2,1)}_{01} & h^{(2,1)}_{11}
\end{pmatrix}, \quad
P_2(Y) = \begin{pmatrix} 0 & 0 & 1 & 0 \\
0 & 0 & 0 & 1\\
h^{(0,2)}_{00} & h^{(0,2)}_{10} & h^{(0,2)}_{01} & 0\\
h^{(1,2)}_{00} & h^{(1,2)}_{10} & h^{(1,2)}_{01} & h^{(1,2)}_{11}
\end{pmatrix}.
\]
The matrices $P_1, P_2$ are called coefficient matrices of a {\em Pfaffian system} 
(an integrable connection) for $F$ (\citet{hgd}).
Note that $P_2$ is obtained from $P_1$ by permutation of $y_1$ and $y_2$.
If we know the values of the components of $\vec F$
at $Y=(y_1, y_2)$, $y_1 \neq y_2$, then 
values at a nearby point $Y+\Delta Y=(y_1 + \Delta y_1 , y_2 + \Delta y_2)$ can be 
approximated by the simple Euler method (i.e.\ linear approximation) as
\begin{align}
\vec F(Y+\Delta Y)&\doteq \vec F(Y) + \Delta y_1 \partial_1 \vec F(Y) + \Delta y_2 \partial_2 \vec F(Y) 
\nonumber \\
&=\vec F(Y) + \Delta y_1 P_1 (Y)\vec F(Y) + \Delta y_2  P_2(Y)\vec F(Y).
\label{eq:euler-approx}
\end{align}

Now suppose that we want to evaluate $F(y_1, y_2)$ at a particular point $(y_1, y_2)$ 
with $y_1\neq y_2$.
If we know $\vec F(Y_0)$ at some point $Y_0=(y_1^{(0)},y_2^{(0)})$,
$y_1^{(0)}\neq y_2^{(0)}$, close to the origin, then 
we can choose an appropriate sequence of points
$Y^{(l)}=(y_1^{(l)},y_2^{(l)})$, $l=0,\dots,L$, such that 
$(y_1, y_2)=(y_1^{(L)},y_2^{(L)})$.
Along the sequence we can use 
\eqref{eq:euler-approx} to update $\vec F(Y^{(l)})$ and finally
the first element of $\vec F(Y^{(L)})$ gives $F(y_1, y_2)$.

Therefore it remains to consider how to obtain the initial values.
Close to the origin we can use the definition \eqref{eq:def-hyper}
of $\hyperF{1}{1}$. If $Y$ is very close to zero, then
we only need zonal polynomials of low orders, whose explicit forms
are known.  
Zonal polynomials up to the third order are as follows; \ 
${\cal C}_{(1)}(Y)={\cal M}_{(1)}(Y)$, 
\begin{equation}
\begin{pmatrix}
{\cal C}_{(2)}(Y) \\
{\cal C}_{(1,1)}(Y) 
\end{pmatrix} = 
\begin{pmatrix}
1 & \dfrac{2}{3}  \\[1.6ex]
0   & \dfrac{4}{3} 
\end{pmatrix}
\begin{pmatrix}
{\cal M}_{(2)} (Y)\\
{\cal M}_{(1,1)} (Y)
\end{pmatrix},
\quad 
\begin{pmatrix}
{\cal C}_{(3)} (Y)\\
{\cal C}_{(2,1)} (Y)\\
{\cal C}_{(1,1,1)} (Y)
\end{pmatrix} = 
\begin{pmatrix}
1 & \dfrac{3}{5} & \dfrac{2}{5} \\[1.6ex]
0   & \dfrac{12}{5} & \dfrac{18}{5} \\[1.2ex]
0  & 0 & 2
\end{pmatrix}
\begin{pmatrix}
{\cal M}_{(3)}(Y) \\
{\cal M}_{(2,1)} (Y)\\
{\cal M}_{(1,1,1)} (Y)
\end{pmatrix}, 
\label{eq:zonal-monomial}
\end{equation}
where ${\cal M}_{\kappa}(Y)$ is the  monomial symmetric polynomial
associated with a partition $\kappa$.
Since $ F(y_1, y_2)$ can be expanded as
\begin{align}
F(y_1, y_2) &= 1 + \frac{(a)_{(1)}}{(c)_{(1)}} {\cal C}_{(1)}(Y) 
 + \frac{1}{2 !} \left( \frac{(a)_{(2)}}{(c)_{(2)}} {\cal C}_{(2)}(Y)
 + \frac{(a)_{(1,1)}}{(c)_{(1,1)}} {\cal C}_{(1,1)}(Y) \right) +\cdots 
\nonumber  \\
 &= 1+ \frac{(a)_{(1)}}{(c)_{(1)}} {\cal M}_{(1)}(Y) 
 +  \frac{(a)_{(2)}}{2 (c)_{(2)}} {\cal M}_{(2)}(Y)
 + \left(\frac{(a)_{(2)}}{3 (c)_{(2)}}  + \frac{2 (a)_{(1,1)}}{ 3 (c)_{(1,1)}} \right) {\cal M}_{(1,1)}(Y) +\cdots, 
\label{eq:expand-to-monomials}
\end{align}
for an example,  $\partial_1 \partial_2 F(0,0)$ is obtained as
\[
 \partial_1 \partial_2  F(0,0) = \dfrac{(a)_2}{3 (c)_2} + \dfrac{2 a (a-\2)}{3 c (c-\2)}.
\]
In a similar manner, we have 
\begin{eqnarray} 
 \partial_1 F(0,0) =  \partial_2 F(0,0) = \frac{a}{c}, \ 
 \partial_1^2 F(0,0) =  \partial_2^2 F(0,0) = \dfrac{(a)_2}{(c)_2}, 
\nonumber \\ 
 \partial_1^2 \partial_2 F(0,0) =  \partial_2^2 \partial_1 F(0,0) = \dfrac{(a)_3}{5 (c)_3} + 
 \dfrac{4 (a)_2 (a-\2)}{5 (c)_2 (c-\2)}.
\label{eq:initial-others}
\end{eqnarray}
These formulae can be obtained by a symbolic mathematics software, such as the
routines for Jack polynomials in {\tt sage} mathematics software system (\citet{sage}).

In order to obtain the initial value $\vec F(Y_0)$ at $Y_0=(y_1^{(0)}, y_2^{(0)})$ close to the origin, 
we can use the approximation
\begin{equation}
\vec F(y_1^{(0)}, y_2^{(0)}) \doteq \vec F(0,0) + y_1^{(0)} \partial_1 \vec F(0,0) + y_2^{(0)}\partial_2 \vec F(0,0).
\label{eq:linear-initial}
\end{equation}

We code the above procedure using deSolve package in the data analysis system R.
We show a simple source program in Appendix \ref{sec:appendix2}.
In addition, since the zonal polynomials  are  easy to evaluate 
for $m=2$, we also evaluate the series expansion of $\hyperF{1}{1}$ up to $k=150$.
As an example, we compute percentage points by two methods for the case of
$n=3$, $\Sigma={\rm diag}(1/2,1/4)$.
The following percentage points for $\ell_1$ agree in two methods
to 6 digits.

\begin{center}
\begin{tabular}{cccc}
50\% & 90\% &95\% & 99\% \\ \hline
1.63785 & 3.54999 &  4.31600&  6.05836
\end{tabular}
\end{center}

\citet{butler-wood-2002} proposed the Laplace approximation for $\hyperF{1}{1}$ and 
 \citet{koev-edelman} proposed efficient algorithms for computing the truncation of 
$\hyperF{1}{1}$.  
For $m=2, n= 30$ and $\Sigma={\rm diag}(1/2,1/4)$, 
Figure~\ref{fig:comp_m=2} shows an illustrative example; the Laplace approximation fails to give the 
upper probability and the approximation by the truncation rapidly converges to zero with partitions of degrees which are not sufficiently large.
The distribution function by the holonomic gradient method is stable and accurate even when $x$ is large.

\begin{figure}[t]
\begin{center}
\includegraphics{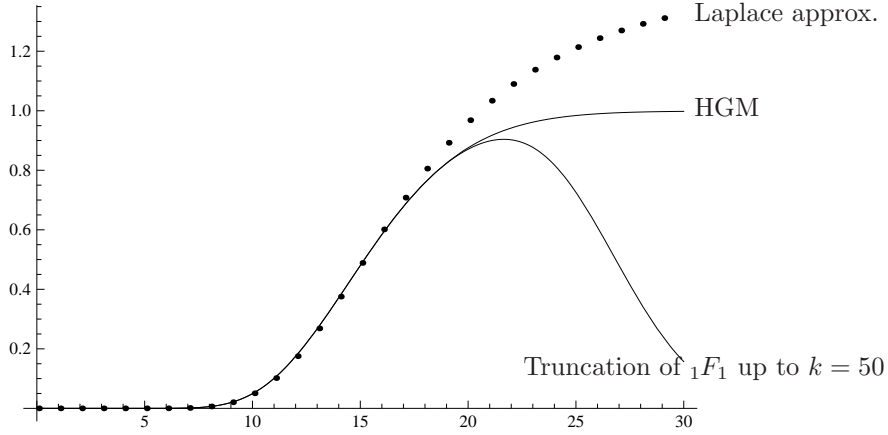}
\rlap{\kern-.5em\raisebox{30ex}{\footnotesize Laplace approx.}}%
\rlap{\kern-.5em\raisebox{23ex}{\footnotesize HGM}}%
\rlap{\kern-6em\raisebox{4ex}{\footnotesize Truncation of $\hyperF{1}{1}$  up to $k=50$}}%

\caption{$m=2, n=30, \Sigma={\rm diag}(1/2,1/4)$} \label{fig:comp_m=2}
\end{center}
\end{figure}

\subsection{Holonomic gradient method for the diagonal line}
\label{subsec:m2-diagonal}

In the previous subsection we assumed $y_1 \neq y_2$ to avoid  singularity of the
differential equations.
However $\hyperF{1}{1}$ itself does not have singularities.  Hence
we should be able to derive some differential equation even for $y=y_1=y_2$.

In \eqref{eq:m2i1} and \eqref{eq:m2i2} we can 
perform the limiting operation $y_1 \rightarrow y_2=y$ using the l'H\^opital rule.
Since $F$ is a symmetric function, at $(y,y)$ we have
\[
\partial_1 F(y,y)=\partial_2 F(y,y).
\]
Also $\partial_1^2 F(y,y)=\partial_2^2 F(y,y)$.
Hence by the
l'H\^opital rule, in \eqref{eq:m2i1} we have
\[
\lim_{y_1 \rightarrow y_2=y} \frac{\partial_1 F - \partial_2 F}{y_1-y_2}
=\partial_2^2 F - \partial_1\partial_2 F = \partial_1^2 F - \partial_1 \partial_2 F.
\]
Then \eqref{eq:m2i1}  for at $(y,y)$ is written as
\begin{equation}
0=\left[y\partial_1^2 + (c-y) \partial_1 + \frac{y}{2} (\partial_1^2 - \partial_1 \partial_2)
- a \right]F
=
\left[\frac{3}{2} y\partial_1^2 + (c-y) \partial_1 - \frac{y}{2} \partial_1 \partial_2 - a \right]F.
\label{eq:fyy0}
\end{equation}
Based on this we derive an ODE 
for $f(y)=F(y,y)$.
Firstly,
\[
f'(y)= 2 \partial_1 F  \text{\  \  or \ \ } \partial_1 F=f'(y)/2.
\]
Secondly,
\[
f''(y)=2 \partial_1^2 F + 2 \partial_1 \partial_2 F.
\]
From \eqref{eq:fyy0}
\[
\frac{3}{2} y\partial_1^2 F=\frac{1}{2}y \partial_1 \partial_2 F- (c-y) \partial_1 F +aF,
\]
and
\begin{align*}
\frac{3}{4} yf''(y) &= \frac{1}{2}y \partial_1 \partial_2 F- (c-y) \partial_1 F +aF
+ \frac{3}{2}y \partial_1 \partial_2 F \nonumber \\
&=2 y \partial_1 \partial_2 F - (c-y) \partial_1 F + aF \nonumber \\
&= 2 y \partial_1 \partial_2 F - \frac{c-y}{2} f'(y) + af(y)
\end{align*}
or
\begin{equation}
\label{eq:fyy3}
\partial_1 \partial_2 F(y,y) = \frac{3}{8} f''(y) + \frac{c-y}{4y}f'(y) - \frac{a}{2y}f(y) .
\end{equation}

Thirdly, 
\begin{equation}
\label{eq:fyy3a}
f'''(y)= 2 \partial_1^3  F + 6 \partial_1^2 \partial_2 F.
\end{equation}
In order to get another relation for $\partial_1^3 F$ and $\partial_1^2 \partial_2 F$,
we differentiate \eqref{eq:m2i1} by $y_2$.  Then by
\begin{equation}
\label{eq:y1y2}
\partial_2 \frac{y_2}{y_1 - y_2} = \partial_2 ( \frac{y_1}{y_1-y_2} - 1)
= \frac{y_1}{(y_1-y_2)^2},
\end{equation}
we obtain the following differential operator  annihilating $F$:
\begin{equation}
y_1 \partial_1^2 \partial_2 + (c-y_1)\partial_1 \partial_2 
+ \frac{1}{2}\frac{y_1}{(y_1 - y_2)^2} (\partial_1 - \partial_2)
+ \frac{1}{2}\frac{y_2}{y_1 - y_2} (\partial_1\partial_2 - \partial_2^2) - a\partial_2.
\label{eq:221}
\end{equation}
Noting  $y_2/(y_1 - y_2)=y_1/(y_1-y_2)-1$ this can be further written as
\begin{align*}
&y_1 \partial_1^2 \partial_2 + (c-y_1)\partial_1 \partial_2 
+ \frac{y_1}{2}\frac{(\partial_1 - \partial_2) + (y_1 - y_2)(\partial_1 \partial_2-\partial_2^2)}{(y_1 - y_2)^2} 
- \frac{1}{2}(\partial_1\partial_2 - \partial_2^2) - a\partial_2\\
&=y_1 \partial_1^2 \partial_2 + (c-1- y_1)\partial_1 \partial_2 
+ \frac{y_1}{2}\frac{(\partial_1 - \partial_2) + (y_1 - y_2)(\partial_1 \partial_2-\partial_2^2)}{(y_1 - y_2)^2} 
+ \frac{1}{2}(\partial_1\partial_2 + \partial_2^2) - a\partial_2.
\end{align*}

We now apply the l'H\^opital rule to
\[
\frac{(\partial_1 - \partial_2) + (y_1 - y_2) (\partial_1\partial_2 - \partial_2^2)}
{(y_1 - y_2)^2} .
\]
We again let $y_1 \rightarrow y_2=y$.
The second derivative of the denominator with respect to  $y_1$ gives $2$.  Now
\begin{align*}
&\partial_1^2 \big( (\partial_1 - \partial_2) + (y_1 - y_2) (\partial_1\partial_2 - \partial_2^2)\big) \\
&\qquad 
= \partial_1 \big( (\partial_1^2 - \partial_1 \partial_2)  + 
(\partial_1\partial_2 - \partial_2^2) + (y_1-y_2)
(\partial_1^2\partial_2 - \partial_1\partial_2^2)
\big)\\
&\qquad =
\partial_1 \big((\partial_1^2 - \partial_2^2) + (y_1-y_2)
(\partial_1^2\partial_2 - \partial_1\partial_2^2)
\big)\\
&\qquad =
(\partial_1^3 - \partial_1\partial_2^2)  
+(\partial_1^2\partial_2 - \partial_1\partial_2^2)
+(y_1-y_2)
(\partial_1^3\partial_2 - \partial_1^2\partial_2^2). 
\end{align*}
Evaluating the right-hand side at $y=y_1=y_2$ and
noting that $\partial_1^2 \partial_2 F=\partial_1 \partial_2^2F$
at $(y,y)$, we just have $\partial_1^3 - \partial_1^2 \partial_2$.
Hence \eqref{eq:221} at $(y,y)$ reduces to 
\begin{align*}
&y \partial_1^2 \partial_2 + (c-1-y)\partial_1 \partial_2 + \frac{y}{4} (\partial_1^3 - \partial_1^2 \partial_2) 
+ \frac{1}{4}(2\partial_1^2 + 2\partial_1 \partial_2)
- a\partial_1\\
&\quad
=
\frac{y}{8} (2\partial_1^3 + 6 \partial_1^2 \partial_2) + (c-1-y)\partial_1 \partial_2
+ \frac{1}{4}(2\partial_1^2 + 2\partial_1 \partial_2)
- a\partial_1,
\end{align*}
where we used $\partial_1 F = \partial_2 F$ at $(y,y)$.
Comparing the right-hand side with 
\eqref{eq:fyy3a} and by \eqref{eq:fyy3}
we obtain
\begin{equation}
\label{eq:m2diag-ode}
\frac{y}{8}f'''(y)+(c-1-y) \big(\frac{3}{8} f''(y) + \frac{c-y}{4y}f'(y) - \frac{a}{2y}f(y)\big)
+ \frac{1}{4}f''(y) - \frac{a}{2} f'(y)=0.
\end{equation}
This equation can be written as
\[
f'''(y)=h_2(y) f''(y) + h_1(y) f'(y) + h_0(y) f(y),
\]
where 
\[
h_2(y)=-\frac{3(c-1-y)}{y} - \frac{2}{y}, \ 
h_1(y)= \frac{4a}{y} - \frac{2(c-y)(c-1-y)}{y^2}, \ 
h_0(y)=\frac{4a(c-1-y)}{y^2}
\]
are rational functions in $y$.
The coefficient matrix for the 
Pfaffian system for a one-dimensional ODE 
is simply the
companion matrix
\[
P=\begin{pmatrix}
0 & 1 & 0 \\
0 & 0 & 1 \\
h_0(y) & h_1(y) & h_2(y)
\end{pmatrix}.
\]

Note that the values of $f$, $f'$, $f''$ and $f'''$ at the origin are 
given by
\begin{gather*}
 f(0) = F(0, 0) = 1, \quad  
 f'(0) = 2 \partial_1 F(0,0) = \dfrac{2 a}{c}, \\
 f''(0) = 2 \partial_1^2 F(0,0)  + 2 \partial_1 \partial_2 F(0,0) 
              = \dfrac{8 (a)_2}{3 (c)_2} +  \dfrac{4 a(a-\frac{1}{2})}{3 c(c-\frac{1}{2})}, \\
f'''(0) = 2 \partial_1^3  F(0,0) + 6 \partial_1^2 \partial_2 F(0,0)
 = 2 \frac{(a)_3}{(c)_3} + 6 \Big(\frac{(a)_3}{5(c)_3} + \frac{4(a)_2 (a-\frac{1}{2})}{5(c)_2 (c-\frac{1}{2})}\Big).
\end{gather*}

As seen above, the computation using the l'H\^opital rule is already tedious for $m=2$.
Actually the computation can be automated by the restriction algorithm for holonomic 
ideals.  This will be explained in Section \ref{subsec:nk-restriction}.

\section{Properties of the Pfaffian system (integrable connection) for a general dimension}
\label{sec:general-dimension}

We now consider our problem for a general dimension.  We fully utilize  Gr\"obner basis theory for
the ring of differential operators.
In this section we only consider the non-diagonal region $\cal X$.
%
Let $K = {\mathbb C}(y_1,\dots,y_m)$ be the
field of rational functions in $y_1,\dots, y_m$ with
complex coefficients.  Further let
\[
R = K\langle \partial_1,\dots,\partial_m\rangle = {\mathbb C}(y_1,\dots,y_m)
\langle \partial_1 \dots,\partial_m\rangle 
\]
be the ring of differential
operators with rational function coefficients (see Appendix of \citet{hgd}).  
Let $I$ denote the left ideal of $R$ generated by $g_1, \dots, g_m$: 
\begin{equation}
\label{eq:ideal}
I=\langle g_1,\dots,g_m \rangle,
\end{equation}
where $g_i$ is given in \eqref{eqn:PDE1f1N}.  

We now prove the following lemma concerning the commutators of $g_1,\dots, g_m$.

\begin{lemma}  For $1\le i\neq j\le m$, 
\label{lem:commutator}
\begin{equation}
\label{eq:commutator}
[g_i, g_j] = -\frac{1}{2} \frac{y_i + y_j}{(y_i - y_j)^2} (g_i - g_j).
\end{equation}
\end{lemma}

A similar result for $\hyperF{2}{1}$ is given in Lemma 9.9 of
\citet{ibukiyama-kuzumaki-ochiai}.  Although they claim that their
Lemma 9.9 follows from a straightforward computation, in fact the computation
for checking \eqref{eq:commutator} is tedious even for $m=2$.
However for $m=2$, 
\eqref{eq:commutator} can be verified by some 
software systems (e.g., \citet{asir}), 
which can handle rings of differential operators.
The following program in {\tt Risa/Asir} 
\begin{center}
\begin{boxedminipage}[ht]{10.5cm}{\scriptsize
\begin{verbatim}
import("names.rr"); import("yang.rr");
yang.define_ring(["partial",[y1,y2]]);
G1=y1*dy1^2+(c-y1)*dy1+(1/2)*(y2/(y1-y2))*(dy1-dy2)-a;
G2=base_replace(G1,[[y1,y2],[y2,y1],[dy1,dy2],[dy2,dy1]]);
G=yang.mul(G1,G2)-yang.mul(G2,G1)+(1/2)*(y1+y2)/(y1-y2)^2*(G1-G2);
printf("G=%a\n",G);
\end{verbatim}
}
\end{boxedminipage}
\end{center}
outputs the result {\tt G=0}.
Therefore in the following proof, assuming that
\eqref{eq:commutator}  holds for $m=2$, we show that it holds for $m>2$.

\begin{proof}
By symmetry we only need to prove the case $i=1, j=2$.
Define $\tilde g_1, \tilde g_2$ 
\begin{align*}
&\tilde g_1 = y_1 \partial_1^2 + (c-y_1)\partial_1 + 
\frac{1}{2}\frac{y_2}{y_1-y_2} (\partial_1 - \partial_2) - a,\\
&\tilde g_2 = y_2 \partial_2^2 + (c-y_2)\partial_2 + 
\frac{1}{2}\frac{y_1}{y_2-y_1} (\partial_2 - \partial_1) - a.
\end{align*}
Then
\[
g_i = \tilde g_i + h_i, \quad h_i=\frac{1}{2}\sum_{k=3}^m  \frac{y_k}{y_i-y_k} (\partial_i - \partial_k),
\quad i=1,2.
\]
We already know
\[
[\tilde g_1, \tilde g_2] =  -\frac{1}{2} \frac{y_1 + y_2}{(y_1 - y_2)^2} (\tilde g_1 - \tilde g_2).
\]
Then 
\[
[g_1, g_2]=[\tilde g_1 + h_1, \tilde g_2 + h_2]=[\tilde g_1, \tilde g_2]
+[h_1, \tilde g_2] + [\tilde g_1, h_2] + [h_1, h_2].
\]
Therefore it suffices to show
\[
[h_1, \tilde g_2] + [\tilde g_1, h_2] + [h_1, h_2] = 
-\frac{1}{2} \frac{y_1 + y_2}{(y_1 - y_2)^2} (h_1 - h_2).
\]

In considering commutators, we only need to look at terms, where a
differential operator actually differentiate rational functions in $y_1,\dots,y_m$.
For example consider  $h_1 \tilde g_2$ in $[h_1, \tilde g_2]$.  In $h_1 \tilde g_2$ the
only relevant term is 
$\partial_1$ in $h_1$ differentiating $y_1/(y_1-y_2)$ in $\tilde g_2$.  Noting
\[
\partial_1 \frac{y_1}{y_2 - y_1} = \partial_1 \big( \frac{y_2}{y_2-y_1} -1 \big)
= \frac{y_2}{(y_2 - y_1)^2},
\]
in $h_1 \tilde g_2$ the relevant terms are
\[
\frac{1}{4} \frac{y_2}{(y_2 - y_1)^2}\sum_{k=3}^m \frac{y_k}{y_1 - y_k} 
(\partial_2 - \partial_1) 
= 
\frac{1}{4} \frac{y_2}{(y_2 - y_1)^2}\sum_{k=3}^m \frac{y_k}{y_1 - y_k} 
\big((\partial_2-\partial_k)  - (\partial_1-\partial_k) \big) .
\]
In $\tilde g_2 h_1$ we need to look at $\partial_1$ in $\tilde g_2$ differentiating 
$y_k/(y_1-y_k)$. Hence we have
\[
\frac{1}{4}\frac{y_1}{y_2 - y_1}\sum_{k=3}^m \frac{y_k}{(y_1 - y_k)^2}(\partial_1 - \partial_k).
\]

Similarly in $[\tilde g_1, h_2]=-[h_2, \tilde g_1]$ the relevant terms are
\[
-\frac{1}{4} \frac{y_1}{(y_1 - y_2)^2}\sum_{k=3}^m \frac{y_k}{y_2 - y_k} 
\big( (\partial_1 - \partial_k) - (\partial_2 - \partial_k)\big)
+\frac{1}{4}\frac{y_2}{y_1 - y_2}\sum_{k=3}^m \frac{y_k}{(y_2 - y_k)^2}(\partial_2 - \partial_k).
\]

Finally in $[h_1,h_2]$ we look at $\partial_k$ differentiating $y_k/(y_i- y_k)$.  Then
the relevant terms are
\[
-\frac{1}{4} \sum_{k=3}^m \frac{y_k}{y_1 - y_k}\frac{y_2}{(y_2 - y_k)^2}(\partial_2 - \partial_k)
+ 
\frac{1}{4} \sum_{k=3}^m \frac{y_k}{y_2 - y_k}\frac{y_1}{(y_1 - y_k)^2}(\partial_1 - \partial_k).
\]
Then the coefficient for $-(\partial_1 - \partial_k)/4$ is
\begin{align*}
&\frac{y_2}{(y_2-y_1)^2} \frac{y_k}{y_1 -y_k}
+\frac{y_1}{y_2 - y_1} \frac{y_k}{(y_1 - y_k)^2}
+\frac{y_1}{(y_1 -y_2)^2} \frac{y_k}{y_2 - y_k}
- \frac{y_k}{y_2 - y_k}\frac{y_1}{(y_1 - y_k)^2}\\
&\quad = \frac{y_1 + y_2}{(y_2 - y_1)^2} \frac{y_k}{y_1-y_k},
\end{align*}
which coincides with the coefficient of $-(\partial_1  - \partial_k)/4$ in
\[
-\frac{1}{2} \frac{y_1 + y_2}{(y_1 - y_2)^2} h_1 .
\]
Similarly the coefficients of $(\partial_2 - \partial_k)$ coincide on both sides.
\end{proof}

We now consider the graded lexicographic term order $\succ$.
The initial term of $g_i$ (without the coefficient $y_i$) is given as
\[
\interm g
_i = \partial_i^2.
\]

We now prove the following theorem.

\begin{theorem} \label{th:gbasis}
For the term order $\succ$, 
$\{g_1,\dots, g_m\}$ is a Gr\"obner basis of $I$ in $R$ and 
the initial ideal is given 
by $\langle \partial_1^2, \dots,\partial_m^2\rangle$.
$I$ is zero-dimensional and the set of standard monomials is given by
the set of square-free mixed derivatives
\[
\{\partial_{i_1} \partial_{i_2}\dots \partial_{i_k} \mid \ 1\le i_1 < \dots < i_k \le m, \ k\le m \},
\]
which has the cardinality $2^m$.
\end{theorem}

\begin{proof}
By Lemma \ref{lem:commutator} and the 
Buchberger's criterion for the ring $R$ 
(cf.\ Theorem 1.1.10 of \citet{saito-sturmfels-takayama}), 
$g_i, i=1,\dots, m,$ form a Gr\"obner basis and 
the initial ideal is given by $\langle \partial_1^2, \dots,\partial_m^2\rangle$.
Let $J=\langle \partial_1,\dots,\partial_m \rangle$. Then
$J^{m+1}\subset 
\langle \partial_1^2, \dots,\partial_m^2 \rangle$.
Hence $I$ is a zero-dimensional ideal.
Furthermore this shows that 
the set of standard monomials is given by the set of
square-free mixed derivatives.
\end{proof}

It follows from  Theorem \ref{th:gbasis} that
there exists a Pfaffian system 
and $2^m \times 2^m$ matrices 
(as $P_i(Y)$ for $m=2$ in the expository section \ref{sec:dim2})
are obtained by the normal form algorithm
in the ring of differential operators $R$.
The matrices are used to numerically solve the associated  ODE.
However, the derivation of the matrices on computer is heavy 
and the obtained matrices are not in a relevant form 
for an efficient numerical evaluation.
Then, we do it {\it by hand} in the sequel.

Consider a higher order derivative
$\partial_1^{n_1} \dots \partial_m^{n_m} F$ 
of $F=\hyperF{1}{1}(a;c;y_1, \dots, y_m)$.
If total degree of differentiation $n=n_1 + \dots + n_m$ is greater than or
equal to $m+1$, then for some $i$ we have $n_i \ge 2$. Then as in the
previous section we can use $g_iF=0$ to decrease the
total degree of differentiation.  Therefore as in 
\eqref{eq:n1n2base}, for each $n_1,\dots,n_m$, there exist $2^m$ 
rational functions 
$h^{(n_1,\dots, n_m)}_{i_1, \dots,i_m}$, $i_j = 0,1$, $j=1,\dots,m,$ such that
\begin{equation}
\label{eq:n1nmbase}
\partial_1^{n_1}\dots \partial_m^{n_m}F = 
\sum_{i_1=0}^1 \dots \sum_{i_m=0}^1 h^{(n_1, \dots, n_m)}_{i_1,\dots, i_m} \partial_1^{i_1} \dots \partial_m^{i_m} F.
\end{equation}

In the holonomic gradient method,  as in the case of $m=2$ in \eqref{eq:euler-approx},
we only need $h^{(n_1,\dots, n_m)}_{i_1, \dots,i_m}$ where
$0\le n_1, \dots, n_m\le 2$ and at most one of $n_1,\dots, n_m$  is two, such as
$h^{(2,1,\dots,1,0, \dots, 0)}_{i_1, \dots,i_m}$.
Define a $2^m$-dimensional vector of square-free  mixed derivatives of $F$ by
$\vec F = (F, \partial_1 F, \partial_2 F, \partial_1 \partial_2 F,  \dots, 
\partial_1\dots \partial_m F)^t$.  
In $\vec F$ the 
elements are lexicographically  ordered, for convenience in programming.
$\partial_i \vec F$ is written as 
\[
\partial_i \vec F = P_i(y) \vec F, \qquad i=1,\dots, m,
\]
where $P_i(y)$, $i=1,\dots, m$, in the Pfaffian system 
are $2^m \times 2^m$ matrices
consisting of  $h^{(n_1,\dots, n_m)}_{i_1, \dots,i_m}$'s.

We now study the form of $h^{(n_1,\dots, n_m)}_{i_1, \dots,i_m}$, 
where $0\le n_1, \dots, n_m\le 2$ and at most one of $n_1,\dots, n_m$  is two.
Denote $[m]=\{1,\dots,m\}$.
For a subset $J\subset [m]$ denote
\[
\partial_J = \prod_{j\in J} \partial_j .
\]
Choose $i\in [m]$ and $J\subset [m]$ such that $i\not\in J$. Write $I=J\cup \{i\}$.
$\partial_J g_i F=\partial_J 0 = 0$, where $g_i$ is in \eqref{eqn:PDE1f1N}.  
Since $i\not \in J$, we can write $\partial_J g_i$ as
\[
y_i \partial_i^2 \partial_J + (c-y_i) \partial_I
+\frac{1}{2} \sum_{k\neq i} \partial_J \big(
\frac{y_k}{y_i - y_k} (\partial_i - \partial_k)\big) - a \partial_J.
\]
For $k\not \in J$ 
\[
\partial_J \big(
\frac{y_k}{y_i - y_k} (\partial_i - \partial_k)\big) = 
\frac{y_k}{y_i - y_k} (\partial_I - \partial_{J\cup \{k\}}).
\]
On the other hand for $k\in J$, by \eqref{eq:y1y2}
\[
\partial_J \big(
\frac{y_k}{y_i - y_k} (\partial_i - \partial_k)\big) = 
\frac{y_k}{y_i - y_k} (\partial_I - \partial_J \partial_k)
+ \frac{y_i}{(y_i - y_k)^2} (\partial_{\{i\}\cup J\setminus\{k\}} - \partial_J).
\]
Here $\partial_J \partial_k$ is not square-free and in fact
\[
\partial_J \partial_k = \partial_k^2 \partial_{J\setminus \{k\}},
\]
which causes recursive application of \eqref{eqn:PDE1f1N}.  In $\partial_J g_i$ we
now separate square-free terms and define
\begin{align*}
 r(i,J; y)&= - \Big[
(c- y_i) \partial_I - a \partial_J + 
\frac{1}{2} \sum_{k\not\in I} \frac{y_k}{y_i - y_k} (\partial_I
- \partial_J\partial_k)\\
& \qquad + \frac{1}{2} \sum_{k\in J} \frac{y_k}{y_i - y_k} \partial_I
+ \frac{1}{2} \sum_{k\in J} \frac{y_i}{(y_i - y_k)^2} (\partial_i \partial_{J\setminus \{k\}}
- \partial_J)\Big],
\end{align*}
where for $J=\emptyset$, reflecting the original $g_i$, 
we define
\[
r(i,\emptyset; y)=-\Big[
(c- y_i) \partial_i - a + \frac{1}{2} \sum_{k\neq i} \frac{y_k}{y_i - y_k} (\partial_i
- \partial_k)\Big].
\]
Then $\partial_i^2 \partial_JF$ is expanded as 
\begin{equation}
y_i \partial_i^2 \partial_J F
=  r(i,J;y)F  + \frac{1}{2} \sum_{k\in J} \frac{1}{y_i - y_k}  (y_k \partial_k^2 \partial_{J\setminus \{k\}})F.
\label{eq:recursion2}
\end{equation}
The use of this recursive expression yields an efficient numerical evaluation
of the matrices of the Pfaffian system.
We keep numerical values of
$\partial_k^2 \partial_{J\setminus \{k\}}F$ in a table and
use them to evaluate 
$\partial_i^2 \partial_J F$ and keep it in the table, again.

We can also apply the recursion to the last term on the right-hand side.  The resulting
expression for $y_i \partial_i^2 \partial_J F$ is given as
\begin{align}
y_i \partial_i^2 \partial_J F &=  r(i,J;y) F
+ \frac{1}{2} \sum_{k_1\in J} \frac{1}{y_i - y_{k_1}} r(k_1,J\setminus\{k\};y) F
\nonumber \\& \quad
+ \frac{1}{4} \sum_{k_1, k_2\in J\atop k_1, k_2: \text{distinct}} 
\frac{1}{(y_i - y_{k_1})(y_{k_1}-y_{k_2})} r(k_2,J\setminus\{k_1, k_2\};y) F
\nonumber \\& \quad
+ \frac{1}{8} \sum_{k_1, k_2,k_3\in J\atop k_1, k_2, k_3: \text{distinct}} 
\frac{1}{(y_i - y_{k_1})(y_{k_1}-y_{k_2})(y_{k_2}-y_{k_3})} 
r(k_3,J\setminus\{k_1, k_2,k_3\};y)F + \dots
\nonumber \\ &\quad 
+ \frac{1}{2^{|J|}}\sum_{k_1, \dots, k_{|J|}\in J\atop k_1, \dots, k_{|J|}: \text{distinct}} 
\frac{1}{(y_i - y_{k_1})(y_{k_1}-y_{k_2})\dots (y_{k_{|J|-1}}-y_{k_{|J|}})} r(k_{|J|},\emptyset;y) F .
\label{eq:expanded}
\end{align}

Now in \eqref{eq:2-1} we write $\Sigma^{-1}/2= \beta=(\beta_1, \dots, \beta_m)$, where
$\beta_1,\dots, \beta_m$ are distinct, 
and define a $2^m$-dimensional vector valued function $\vec G$ in a scalar $x$ by
\[
\vec G(x)=\exp( -x\sum_{i=1}^m \beta_i) x^{mn/2} \vec F(\beta x).
\]
Then $\vec G$ satisfies the ODE
\begin{equation} \label{eq:odeG}
\frac{d\vec G}{dx}= \left( - (\sum_{i=1}^m \beta_i) I_{2^m} + \frac{mn}{2x} I_{2^m}
 + \sum_{i=1}^m P_i(\beta x) \beta_i\right) \vec G,
\end{equation}
where $I_{2^m}$ is the $2^m \times 2^m$ identity matrix. 
We denote the right-hand side as $P_\beta \vec G$.  We now prove the following theorem,
which is important for guaranteeing stability of ODE at $x=+\infty$.

\begin{theorem}
\label{th:takayama1}
As $x\rightarrow\infty$
\[
P_\beta=A_0 + O(1/x),
\]
where $A_0$ only depends on $\beta$ and the $2^m$ eigenvalues of $A_0$ are given
as $-e_1 \beta_1 - \dots - e_m \beta_m$, where $(e_1,\dots, e_m)\in \{0,1\}^m$.
\end{theorem}

\begin{proof}
Note that $y_1=\beta_1 x,\dots, y_m=\beta_m x = O(x)$.
Divide \eqref{eq:expanded} by $y_i = \beta_i x$.  Then on the right-hand
side of \eqref{eq:expanded}, the only constant order term is $\partial_I$ in $r(i,J;y)$.
Now 
\begin{align*}
\frac{d}{dx} \partial_I F(\beta x)&= \sum_{i=1}^m \beta_i \partial_i \partial_I F(\beta x)\\
&=\sum_{i\in I} \beta_i \partial_i^2 \partial_{I\setminus \{i\}} F(\beta x)   
+ \sum_{i\not \in I} \beta_i \partial_{I\cup \{i\}} F(\beta x)\\
&= \big(\sum_{i\in I} \beta_i\big)  \partial_I F(\beta x)  + O(1/x)
 + \sum_{i\not \in I} \beta_i \partial_{I\cup \{i\}} F(\beta x).
\end{align*}
This implies that the $I$-th diagonal element of 
$A_0$ is given 
\[
- \sum_{i=1}^m \beta_i + \sum_{i\in I} \beta_i
= - \sum_{i\not\in I} \beta_i.
\]
Furthermore the $(I, I\cup \{i\})$-element of $A_0$ is $\beta_i$.
Other elements of $A_0$ are zeros.  Hence $A_0$ is an upper triangular
matrix with diagonal elements $-\sum_{i\not\in I} \beta_i$, $I\subset[m]$.
The theorem holds because the  diagonal elements of an upper triangular
matrix are its eigenvalues.
\end{proof}


\section{Some results of symbolic computation}
\label{sec:symbolic-computation}

In this section we present some results on symbolic computation for
the initial values (cf.\ \eqref{eq:initial-others})
and the restriction for diagonal regions (cf.\ Section \ref{subsec:m2-diagonal}).
We omit writing down the fully expanded form of 
\eqref{eq:expanded}, since the recursive  formula
\eqref{eq:recursion2} can be directly used in our implementation of holonomic gradient method.


\subsection{Initial values}
\label{subsec:initial-values}
Initial values for our holonomic gradient method can be obtained by
expressing $\hyperF{1}{1}$ in terms of monomial symmetric polynomials as in
\eqref{eq:expand-to-monomials}.  We denote the relation between the
zonal polynomials and the monomial symmetric polynomials in \eqref{eq:zonal-monomial}
as 
\[
{\cal C}_\kappa(Y) = \sum_{\lambda \dominatedby \kappa} c_{\kappa,\lambda} {\cal M}_{\lambda}(Y), 
\]
where $\lambda \dominatedby \kappa$ means that $\lambda$ is dominated by $\kappa$, i.e.\ 
$\sum_{i=1}^{s} \lambda_i \leq \sum_{i=1}^s \kappa_i$
for all $s$.  
Then 
$\hyperF{1}{1}$ is expressed as
\[
\hyperF{1}{1}(a;c;Y)=\sum_{k=0}^\infty \sum_{\lambda \partitionof k}  q_\lambda(a,c) {\cal M}_\lambda(Y),
\quad
q_{\lambda}(a,c)=
\sum_{\kappa \partitionof k, \kappa \dominates \lambda}
\frac{(a)_{\kappa} c_{\kappa,\lambda}}{(c)_{\kappa}k!}.     
\]

A recurrence relation for $c_{\kappa,\lambda}$'s is given by
\citet{james-1968} (see also (14) in Section 7.2.1 of \citet{muirhead-book} and
Section 4.5.4 of \citet{takemura-zonal}), which can be used to compute $q_\kappa(a,c)$.
However James' recurrence relation works for each ${\cal C}_\kappa$ separately.
Recently \citet{koev-edelman} gave a much improved algorithm based on
recursive relations among the values of zonal polynomials for 
$m$ variables and $m-1$ variables.
For our implementation
of holonomic gradient method, we adapted
Koev-Edelman's recurrence relation also for derivatives of $\hyperF{1}{1}$ to
evaluate the initial values.

Close to the origin, we can use rough initial values given by the linear approximation as in 
\eqref{eq:linear-initial}.
Then we only need $\kappa=(k_1,\dots,k_l)$ such that $k_1=\dots=k_l=1$ or
$k_1=2$, $k_1=\dots=k_l=1$.  Some  $q_\lambda(a,c)$'s for small $\lambda$'s are as follows.
\begin{align*}
&q_{\emptypartition}=1, \ \ q_{(1)} = \frac{a}{c}, \ \  q_{(2)}=\frac{(a)_2}{2 (c)_2},\ \ 
q_{(1,1)}= \frac{(a)_2}{3 (c)_2} + \frac{2(a)_{(1,1)}}{3 (c)_{(1,1)}}, \ \ 
q_{(2,1)}=\frac{(a)_3}{10 (c)_3} + \frac{2 (a)_{(2,1)}}{ 5 (c)_{(2,1)}},\\
& q_{(1,1,1)}=\frac{(a)_3}{15(c)_3} + \frac{3(a)_{(2,1)}}{5(c)_{(2,1)}}
  + \frac{(a)_{(1,1,1)}}{3(c)_{(1,1,1)}},\  
q_{(2,1,1)}=\frac{(a)_4}{70(c)_4} + \frac{4(a)_{(2,2)}}{45(c)_{(2,2)}} + \frac{11(a)_{(3,1)}}{63(c)_{(3,1)}}
+ \frac{2(a)_{(2,1,1)}}{9(c)_{(2,1,1)}},
\end{align*}
where $\emptypartition$ stands for the unique partition of zero. 
Write $(1^k)=(1,\dots,1), (2,1^{k-2})=(2,1,\dots,1)$, which are partitions of $k$. 
Given the above quantities,
the linear approximation of $\partial_1 \dots\partial_l F(Y)$, $0\le  l \le m$,  
for  $Y=(y_1,\dots,y_m)$ close to the origin is expressed as 
\begin{equation}
\label{eq:linear-initial-m}
\partial_1\dots \partial_l F(Y) \doteq  q_{(1^l)}(a,c) + 2q_{(2,1^{l-1})} (a,c) (y_1 + \dots + y_l) + 
q_{(1^{l+1})}(a,c) (y_{l+1} + \dots + y_m), 
\end{equation}
where for $l=0$ the second term on the right-hand side is zero and
for $l=m$ the third term is zero.
We found that initial values by \eqref{eq:linear-initial-m} are 
practical enough for $m\le 5$.

In fact, by Lemma 1 in Section 4.5.2 of \cite{takemura-zonal}
and by Proposition 7.3 of \cite{stanley-1989am},
$q_{(1^k)}(a,c)$ and $q_{(2,1^{k-2})}(a,c)$ are explicitly written 
as follows:
\begin{align*}
q_{(1^k)}(a,c)&=2^k k! \sum_{\kappa \partitionof k} \frac{\prod_{1\le i<j\le l(\kappa)} (2k_i - 2k_j - i +j)}
{\prod_{i=1}^{l(\kappa)} (2k_i + l(\kappa)-i)!} \frac{(a)_\kappa}{(c)_\kappa}, \\
q_{(2,1^{k-2})}(a,c)&=2^k (k-2)!\sum_{\kappa \partitionof k} \frac{\prod_{1\le i<j\le l(\kappa)} (2k_i - 2k_j - i +j)}
{\prod_{i=1}^{l(\kappa)} (2k_i + l(\kappa)-i)!} 
\big(\binom{k}{2} + \sum_{i=1}^{l(\kappa)} k_i (k_i - i)\big)
\frac{(a)_\kappa}{(c)_\kappa},
\end{align*}
where $l(\kappa)$ is the length (number of non-zero parts) of $\kappa=(k_1,\dots,k_{l(\kappa)})$.

For larger values of $m$ we need higher order terms for initial values.
For two partitions $\mu, \lambda$, we write $\mu \subset \lambda$ to denote 
$\mu_i \le  \lambda_i$ for all $i$.
For two partitions $\kappa, \nu$, 
we denote
by $\kappa\concat\nu$ 
the concatenation of $\kappa$ and $\nu$ obtained 
from $(\kappa_1,\nu_1,\kappa_2,\nu_2,\ldots)$
by sorting.
Consider a rectangular partition $\tau=(t,\dots,t)=(t^l) \partitionof tl$.
For $\tau=(t^l)$ and $\lambda  \supset \tau$
we define
\[
I(\lambda,\tau) = \{(\kappa,\nu) \mid
\kappa\concat\nu= \lambda,\ 
\tau \subset \kappa, \ \kappa_{l+1}=0\}.
\]
Consider $\der^\mu {\cal M}_\lambda(Y)=\der_1^{\mu_1}\der_2^{\mu_2}\cdots \der_m^{\mu_m}
{\cal M}_\lambda(Y)$.
Note that $\der^\mu {\cal M}_\lambda(Y)=0$, if
$\mu \not\subset \lambda$. 
For a rectangular  $\tau=(t^l)$
we can calculate $\der^\tau {\cal M}_\lambda(Y)$
by the following lemma.
\begin{lemma}
\label{lemma:monomial-d}
For $\tau=(t^l) \subset \lambda$
\begin{align*}
\der^\tau  {\cal M}_\lambda(Y)
=\sum_{ (\kappa,\nu) \in I(\lambda,\tau)}
\frac{\kappa!}{(\kappa-\tau)!}
{\cal M}_{\kappa-\tau}(y_1,\dots,y_l)
{\cal M}_{\nu}(y_{l+1},\ldots,y_m),
\end{align*}
where 
$\gamma!=\prod_i (\gamma_i!)$ for a partition $\gamma$, and
$\kappa-\tau=(\kappa_1-t,\ldots,\kappa_l-t)$.
\end{lemma}

Proof is straightforward and omitted.
Using this lemma we have the following proposition.
\begin{proposition}
\label{prop:hgf-d}
For a rectangular partition $\tau=(t^l)$, 
\begin{equation}
\label{eq:rectangular-derivative-of-f11}
\der^\tau \hyperF{1}{1}(a;c;Y)=
\sum_{k=tl}^{\infty}
\sum_{\substack{\lambda \partitionof k,\\ \tau\subset\lambda}}
q_\lambda(a,c)
\sum_{ (\gamma,\nu) \in I(\lambda,\tau)}
\frac{\gamma!}{(\gamma-\tau)!}
{\cal M}_{\gamma-\tau}(y_1,\ldots,y_{l})
{\cal M}_{\nu}(y_{l+1},\ldots,y_m).
\end{equation}
\end{proposition}

For our initial values we only need to consider $\tau=(1^l)$.
We obtain \eqref{eq:linear-initial-m} if 
we only look at linear terms on the right-hand side of \eqref{eq:rectangular-derivative-of-f11}.
Note that since $\hyperF{1}{1}(a;c;Y)$ is a symmetric function in 
$y_1,\ldots,y_m$, other derivatives are obtained by permutation of $y_1, \dots, y_m$.

Although  \eqref{eq:rectangular-derivative-of-f11} 
only gives derivative with respect to a rectangular partition $\tau=(t^l)$,
we can obtain other derivatives $\partial_1^{\mu_1} \dots \partial_l^{\mu_{h}} \hyperF{1}{1}(a;c;Y)$,
$\mu_1 \ge \dots\ge \mu_{h}$,  by repeated application of 
\eqref{eq:rectangular-derivative-of-f11} for different values of $l$'s.

\subsection{Restriction to diagonal regions}
\label{subsec:nk-restriction}

As mentioned at  the end of Section \ref{subsec:m2-diagonal}, the tedious operation involving
the l'H\^opital rule for the diagonal region  can be performed by
the restriction algorithm for holonomic  ideals.
The following program in {\tt Risa/Asir} for $m=2$
\begin{center}
\begin{boxedminipage}[ht]{12cm}{\scriptsize
\begin{verbatim}
import("names.rr"); import("nk_restriction.rr");
G1=y1*dy1^2 + (c-y1)*dy1+(1/2)*(y2/(y1-y2))*(dy1-dy2)-a;  G1=red((y1-y2)*G1);
G2=base_replace(G1,[[y1,y2],[y2,y1],[dy1,dy2],[dy2,dy1]]);
F=base_replace([G1,G2],[[y1,y],[y2,y+z2],[dy1,dy-dz2],[dy2,dz2]]);
A=nk_restriction.restriction_ideal(F,[z2,y],[dz2,dy],[1,0] | param=[a,c]);
\end{verbatim}
}
\end{boxedminipage}
\end{center}  
produces the output
{\footnotesize
\begin{verbatim}
-y^2*dy^3+(3*y^2+(-3*c+1)*y)*dy^2+(-2*y^2+(4*a+4*c-2)*y-2*c^2+2*c)*dy-4*a*y+(4*c-4)*a
\end{verbatim}
}
\noindent
This is the same as \eqref{eq:m2diag-ode}.  Adapting the above program for $m=3$, we obtain
\begin{align*}
&y^3 f''''(y) +(-6y^3+(6c-4)y^2)f'''(y)\\
& \quad + (11y^3+(-10a-22c+18)y^2+(11c^2-17c+4)y)f''(y) \\
& \quad +(-6y^3+(30a+18c-18)y^2+((-30c+34)a-18c^2+34c-12)y\\
& \qquad \qquad \qquad +6c^3-16c^2+10c)f'(y) \\
& \quad + (-
18ay^2+(9a^2+(36c-51)a)y+(-18c^2+48c-30)a)f(y) = 0.
\end{align*}
For $m=4$, we found that the computation by 
{\tt Risa/Asir} takes too much time and memory.  

We conjecture that 
the ideal $I$ generated by
$ \prod_{j\neq i} (y_i-y_j) g_i$, $i=1, \ldots, m$
in the Weyl algebra 
$D_m$ is an holonomic ideal.
In fact, the conjecture can be checked for small dimensions $m$ on a computer.
See the Appendix \ref{sec:appendix}.
If $I$ is a holonomic ideal,
then 
$$ J=\left( I + (y_1-y_2) D_m + (y_1-y_3) D_m + \cdots + (y_1-y_m) D_m \right)
\cap \C \langle  y_1, \partial_{y_1} \rangle $$
is not $0$ and is an holonomic ideal in $D_1$
by the theorem of Bernstein (see, e.g., the Chapter 5 of 
\cite{saito-sturmfels-takayama}).
The generators of $J$ is 
ordinary differential equations for the function restricted to 
the diagonal $y_1 = \dots = y_m$.
Thus, the holonomicity implies the existence of the diagonal ordinary
differential equation.
The ideal $J$ can be obtained by Oaku's algorithm (\cite{oaku-1997}) based on Gr\"obner bases
and the {\tt Risa/Asir} package {\tt nk\_restriction} uses this algorithm.


\section{Numerical experiments} 
\label{sec:numerical}
We implemented the holonomic gradient method  in a straightforward manner.
Our source programs in the language C are available from\\ 
{\tt http://www.math.kobe-u.ac.jp/OpenXM/Math/1F1}.

The updating step of the holonomic gradient method was implemented using the recursive relation
\eqref{eq:recursion2} for a general dimension.  For initial values we adapted
\citet{koev-edelman} for derivatives of $\hyperF{1}{1}$ as discussed in 
Section \ref{subsec:initial-values}.

The accuracy of the holonomic gradient method can be simply checked by looking at
the numerical convergence  $\Pr[\ell_{1} < x] \rightarrow 1$ as $x\rightarrow\infty$. 
This is because the initial values are evaluated at  small $x>0$ and 
$\Pr[\ell_{1} < x]$ at large $x$ is obtained after many updating steps.
This is another  advantage of our method.

Also we can use the following simple bounds for the upper tail probability for the purpose of checking.
Let $\Pr[\ell_{1} < x | \Sigma] $ denote the probability under the
covariance matrix $\Sigma$.  
Consider $\Sigma=\diag(\sigma_1^2,\dots,\sigma_m^2)$, $\sigma_1^2 \ge \dots \ge \sigma_m^2$.
Then by standard stochastic ordering consideration, we have
\begin{align}
\Pr[\ell_{1} < x | \diag(\sigma_1^2,\dots,\sigma_1^2)] 
&\le \Pr[\ell_{1} < x | \diag(\sigma_1^2,\dots,\sigma_m^2)]\nonumber \\
&\le \Pr[\ell_{1} < x | \diag(\sigma_1^2,0,\dots,0)].
\label{eq:two-bounds}
\end{align}
The upper bound coincides with the cumulative probability of chi-square distribution 
with $n$ degrees of freedom (cf., \cite{sugiyama-1972ajs},\cite{takemura-sheena-2005}).
Accurate approximation for the
lower bound $\Pr[\ell_{1} < x | \sigma_1^2 I_m]$ is given by the
tube method (\cite{kuriki-takemura-2001}, \cite{kuriki-takemura-2008}).

We first consider the case $m=5, n=7$, 
$\Sigma^{-1}/2= \beta=(1,2,3,4,5)$.  For $x=20$, the two bounds in
\eqref{eq:two-bounds} are given as
$0.9996034$ and $0.9999987$.  With the initial value of $x_0=0.01$ and
step size $0.0001$ we obtained
\[
\Pr[\ell_{1} < 20] =0.999972.
\]
The cumulative distribution function for this case is plotted on the left part
of Figure \ref{fig:1}.

Next we consider the case $m=10, n=12$,  $\beta=(1,2, \dots, 10)$.
For $x=30$, the bounds are
$(0.99866943,  0.99999998)$. For generation of initial values 
it takes about 20 seconds for approximating  $\hyperF{1}{1}$ and its derivatives
up to the degree 20 with an Intel  Core i7 CPU. With the initial value of
$x_0=0.2$ and step size 0.001, we obtain
\[
\Pr[\ell_{1} < 30] =0.999545
\]
in about 75 seconds.  The cumulative distribution function for this case is plotted 
on the right part of Figure \ref{fig:1}. 
We see that enough accuracy is obtained even for $m=10$ within practical amount of time.

\begin{figure}[htbp]
\begin{center}
\includegraphics[width=7cm]{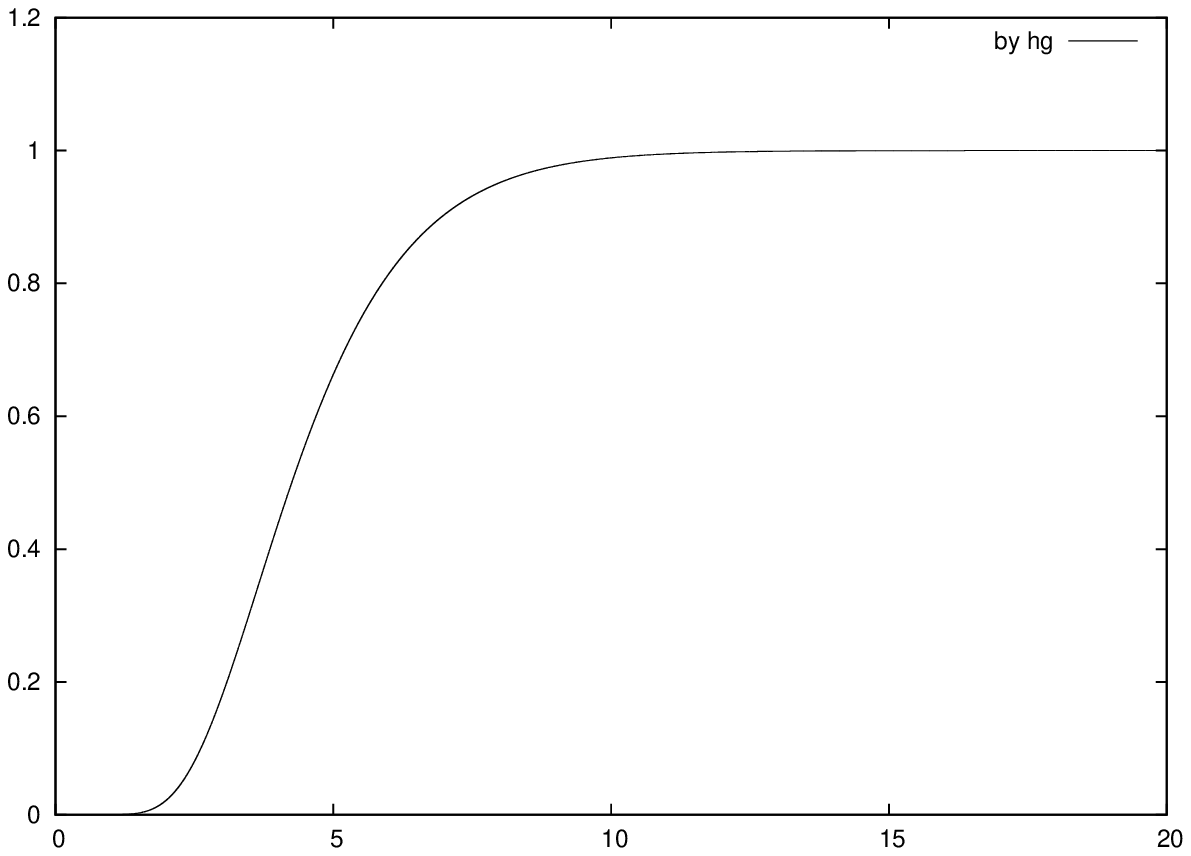}
\includegraphics[width=7cm]{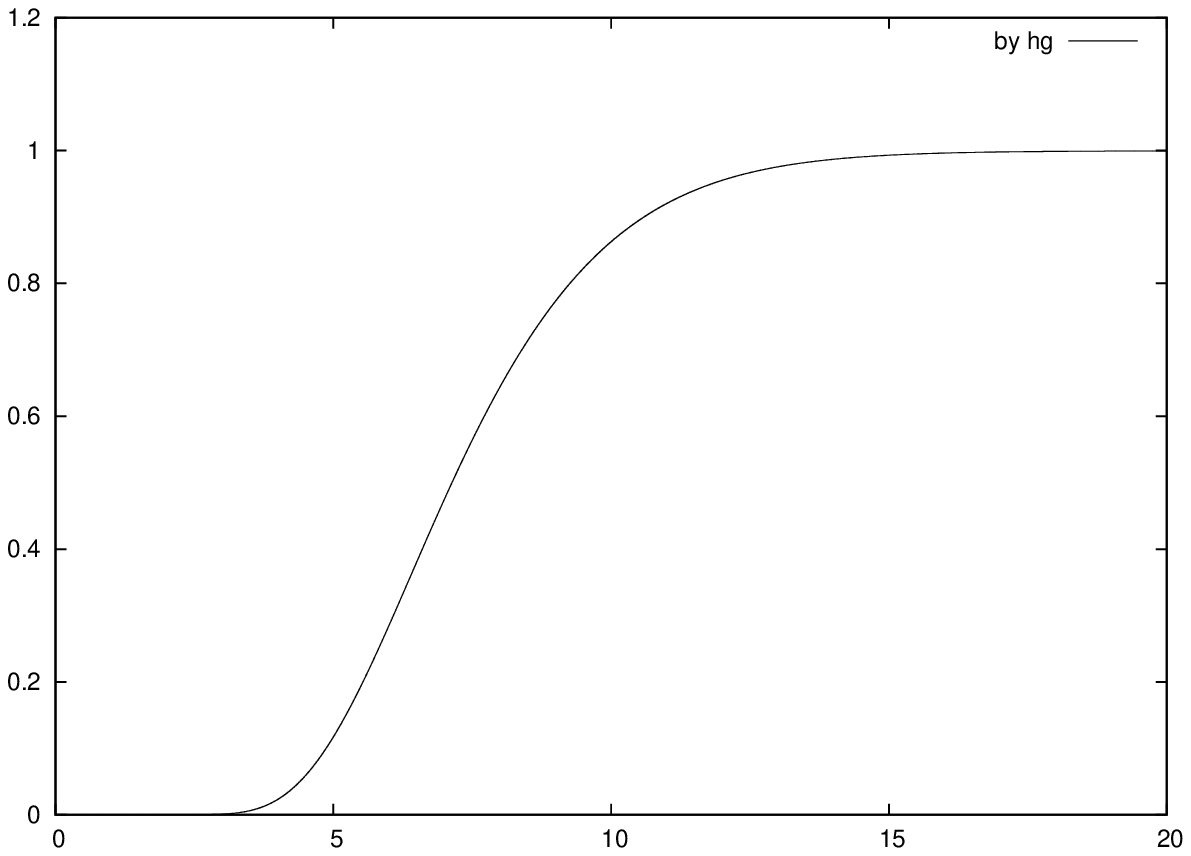}
\caption{Cumulative distributions for $m=5$ and $m=10$}
\label{fig:1}
\end{center}
\end{figure}

The Laplace approximation fails to give a probability for the above two cases too
as in $m=2$ (see Section \ref{subsec:m2-general}); it exceeds one.

The complexity of numerically solving the ODE for $\vec G$ (\ref{eq:odeG})
is
\[
O(m 2^m) \times \mbox{(steps of the Runge-Kutta method with a prescribed
precision)}  .
\]
In fact, since the matrix $P_i(\beta x)$ has sparsity,
each vector $P_i(\beta x) \vec G(x)$, which has $2^m$ elements, can be evaluated in
$O(2^m)$ steps at $x$ from the values of $\vec G(x)$
by utilizing (\ref{eq:recursion2}).

\section{Discussion of open problems}
\label{sec:discussion}

The holonomic gradient method (\cite{hgd}) gives a general algorithm for obtaining the partial differential equations
satisfied by parametrized definite integrals such as the normalizing constant of a family of 
probability distributions.  In fact, in \citet{hgd} and \citet{so3} we used the holonomic gradient method
for deriving the partial differential equations of the normalizing constants and for maximum likelihood estimation for distributions in directional statistics. For
the case of $\hyperF{1}{1}$, the partial differential equations were already derived by \citet{muirhead-1970} more than 40 years ago.
Our use of those partial differential equations for numerical evaluation of $\hyperF{1}{1}$ is very straightforward 
as discussed  in Section \ref{sec:dim2} for the two dimensional case.  
Yet, from the viewpoint of holonomic functions, the partial differential equations of \citet{muirhead-1970}
present many interesting open problems.  

One important question is to obtain the
ordinary and partial differential equations for the diagonal case as discussed in Section
\ref{subsec:m2-diagonal} for the case of $m=2$.  For a general dimension $m>2$, 
it is desirable to be able to handle
various patterns of diagonalization, such as the two-block diagonalization $y_1=\dots=y_l > 
y_{l+1} = \dots =y_m$.   A direct ``by hand'' calculation using  the l'H\^opital rule 
becomes quickly infeasible when we increase $m$.  Also the use of the restriction
algorithm for holonomic ideals is limited by computational complexity.  It is in fact
a very heavy algorithm.  Currently the \verb|nk_restriction| routine 
of {\tt Risa/Asir} in Section \ref{subsec:nk-restriction} takes too much time for $m\ge 4$.
One possibility is to follow the approach in \citet{muirhead-1970} and \citet{sugiyama-takeda-fukuda-1998},
where differentiation with respect to elementary symmetric functions
of the roots of $Y$ are considered.
As discussed in Section \ref{subsec:nk-restriction}, we conjecture that
the ideal $I$ generated by $ \prod_{j\neq i} (y_i-y_j) g_i$, $i=1, \ldots, m$
in the Weyl algebra  $D_m$ is an holonomic ideal.
Holonomicity guarantees the existence of partial 
differential equations for diagonal regions.

Another question is to consider the asymptotics for $\Pr[\ell_1 \ge x]=1-\Pr[\ell_{1} <
x]$ as $x\rightarrow\infty$. As mentioned in the previous section, this
tail probability can be approximated by the tube method
(\cite{kuriki-takemura-2001}, \cite{kuriki-takemura-2008}).
One theoretical problem in applying the tube method is that
only the approximation for the tail probability itself has been justified and
the justification of its derivatives has to be proved.
However it is obvious that the current approach of taking the initial values close to the
origin causes difficulty in precision for the extreme upper tail probability, in the
case we want to evaluate the small probability $\Pr[\ell_1 \ge x]$. Hence it is 
desirable to be able to use tube formula approximation as the initial values at $x=\infty$.

From computational viewpoint, our holonomic gradient method has exponential complexity
in the dimension $m$.  We need to keep the $2^m$-dimensional numerical vector $\vec F$
in memory at each step of the iteration.  For $m=20$, the dimension of the vector
is about one million.  Hence we do not expect that the current implementation of the
holonomic gradient method works for $m=20$.   It might be possible to improve our
current implementation by fully exploiting the fact that $\hyperF{1}{1}$ is a symmetric
function in $Y$.

\appendix
\section{Holonomicity for dimension two}  \label{sec:appendix}
In the theory of holonomic functions, the holonomicity of the left ideal
generated by the set of partial differential operators is an important question.
In fact, the existence of the ordinary differential equation with 
polynomial coefficients for the function restricted to the diagonal region
follows from the holonomicity.
Holonomicity of the ideal generated by $g_1, g_2$ in the two-dimensional case can
be verified by Gr\"obner basis computation.  Here we present this result.
As to a general introduction to holonomic ideals and Gr\"obner bases, 
we refer to the Chapter 1 of \citet{saito-sturmfels-takayama}.

Note that the holonomicity on the non-diagonal region ${\cal X}$
follows from Theorem \ref{th:gbasis} because the zero set of 
$y_i \xi_i^2 = 0$, $i=1, \ldots, m$ contains 
the characteristic variety
on ${\cal X}$. 
The holonomicity on ${\cal X}$ can also be proved by an analogous method with \citet{ibukiyama-kuzumaki-ochiai}.

Let $D_2$ be the second Weyl algebra.
For 
$
P=
\sum_{k=0}^{d} 
\sum_{\alpha_1+\alpha_2=k}
f_{\alpha_1,\alpha_2}(y_1,y_2) \der_1^{\alpha_1} \der_2^{\alpha_2} 
\in D_2,
$
we define $\initial(P)$ by
\begin{align*}
\initial(P) =  \sum_{\alpha_1+\alpha_2=d}f_{\alpha_1,\alpha_2}(y_1,y_2) \xi_1^{\alpha_1} \xi_2^{\alpha_2}  \in \CC[y_1,y_2,\xi_1,\xi_2],
\end{align*}
where we assume 
that $f_{\alpha_1,\alpha_2}(y_1,y_2)\in \CC[y_1,y_2]$ 
and that $f_{\alpha_1,\alpha_2}(y_1,y_2)\neq 0$ 
for some $\alpha_1,\alpha_2$  with $\alpha_1+\alpha_2=d$.
For a left ideal $I$ of $D_2$,
the characteristic variety $\ch(I)$
is defined by 
\begin{align*}
\ch(I)=\{(y_1,y_2,\xi_1,\xi_2)\in \CC^{2\cdot 2} \mid 
\forall P \in I, \initial(P)(y_1,y_2,\xi_1,\xi_2)=0\}.
\end{align*}
It is a basic fact that the dimension of the 
characteristic variety $\ch(I)$
of the proper left ideal $I$ of $D_2$
is greater than or equal to $2$.
A left ideal $I$ of $D_2$ is called {\em holonomic}
if the dimension of the characteristic variety $\ch(I)$ equals $2$.

Let  $P_1=(y_1-y_2)g_1, P_2 = (y_2 -y_1)g_2$
and let  $I$ be the ideal of $D_2$ generated by $P_1$ and $P_2$.
We will show that $I$ is holonomic.
Let $S=y_2\der_2^2 P_1+y_1\der_1^2 P_2+c(\der_2 P_1+\der_1 P_2)\in I$.
By direct calculation we have
\begin{align*}
S=&
(  y_1^2 y_2
- y_1 y_2^2
+ \frac{y_1^2}{2} 
- 2 y_1 y_2) \der_1^2 \der_2 
+(
- y_1^2 y_2
+ y_1 y_2^2
- 2 y_1 y_2
+\frac{y_2^2}{2} ) \der_1 \der_2^2 
-\frac{y_1^2 }{2} \der_1^3 
- \frac{y_2^2}{2}  \der_2^3 
\\
&+( a y_1^2 
- a y_1 y_2 
-  \frac{3cy_1 }{2} 
- y_1) \der_1^2 
+(
- y_2
+2a y_2
-ay_1 y_2
+a y_2^2 
- \frac{3cy_2}{2} ) \der_2^2  
\\
&+(
- cy_1^2
+ 2 cy_1 y_2
-cy_2^2
+ 4 y_1 y_2 
- \frac{3cy_1 }{2} 
-\frac{3cy_2}{2} 
+ y_1 
+ y_2) \der_1 \der_2 \\
&+ (ac y_1 
- ac y_2 
+ 2 a y_1
+ c y_1 
- c^2) \der_1 
+(
- c^2
- ac y_1 
+ac y_2 
+c  y_2) \der_2
+ 2 a c
.
\end{align*}
Hence
\begin{align*}
\initial(S)=&
(  y_1^2 y_2
- y_1 y_2^2
+ \frac{y_1^2}{2} 
- 2 y_1 y_2) \xi_1^2 \xi_2 
+(
- y_1^2 y_2
+ y_1 y_2^2
- 2 y_1 y_2
+\frac{y_2^2}{2} ) \xi_1 \xi_2^2 
-\frac{y_1^2 }{2} \xi_1^3 
- \frac{y_2^2}{2}  \xi_2^3.
\end{align*}
Consider 
the ideal $J$ of $\CC[y_1,y_2,\xi_1,\xi_2]$ generated by 
$\initial(P_1)$, $\initial(P_2)$ and $\initial(S)$.
The following is
a Gr\"obner base of $J$ with respect to 
the graded reverse lexicographic order with $y_1 > y_2 >\xi_1 >\xi_2$:
\begin{gather*}
\{y_2^2 \xi_1^3 \xi_2^2 + 3 y_2^2 \xi_1^2 \xi_2^3 + 3 y_2^2 \xi_1
 \xi_2^4 + y_2^2 \xi_2^5, 
y_1 y_2 \xi_1^3 + 3 y_1 y_2 \xi_1^2 \xi_2 
+ 3 y_2^2 \xi_1 \xi_2^2 + y_2^2 \xi_2^3, \\
y_1^2 \xi_1^2 -y_1 y_2 \xi_1^2, 
y_1 y_2 \xi_2^2 - y_2^2 \xi_2^2
\}.
\end{gather*}
Thus the Krull dimension of $J$ is $2$.
Since $\{P_1,P_2,S\} \subset I$,
the characteristic variety $\ch(I)$ of $I$
is contained in the zero set of $J$.
This implies that
the dimension of $\ch(I)$
does not exceed $2$
and hence the dimension of $\ch(I)$ is equal to $2$.
Therefore $I$ is holonomic for $m=2$.

\section{R source program for dimension two}  \label{sec:appendix2}
The following program in data analysis system  R 
implements holonomic gradient method for $m=2$ of Section \ref{subsec:m2-general}
based on deSolve add-on package for R.

\scriptsize
\begin{verbatim}
library(deSolve)
m <- 2    # dimension
n <- 3    # degrees of freedom
x <- 4.31600  #  specify x.   We evaluate  Pr( l1 < x )
b1 <- 1; b2 <- 2 #  (b1,b2) =  (1/2) diag(Sigma^{-1})
a <- (m+1)/2; c <- (n+m+1)/2;  totalsteps <- 10000; stepsize <- x/totalsteps

# h's
h2000 <- function(y1,y2) a/y1
h2010 <- function(y1,y2) -(c-y1)/y1 - y2/(2*y1*(y1-y2))
h2001 <- function(y1,y2) y2/(2*y1*(y1-y2))
h1200 <- function(y1,y2) a/(2*y2*(y2-y1)) 
h1210 <- function(y1,y2) 3/(4*(y2-y1)^2) + a/y2 - (c-y1)/(2*y2*(y2-y1))
h1201 <- function(y1,y2) -3/(4*(y2-y1)^2)
h1211 <- function(y1,y2) -(c-y2)/y2 - y1/(2*y2*(y2-y1))

#initial values
x1 <- b1*stepsize; x2 <- b2*stepsize
fi <- c(a/c, (a*(a+1))/(c*(c+1)), (a*(a+1))/(3*c*(c+1)) + (2*a*(a-1/2))/(3*c*(c-1/2)),
       (a*(a+1)*(a+2))/(5*c*(c+1)*(c+2)) + (4*a*(a+1)*(a-1/2))/(5*c*(c+1)*(c-2/1)))
fi <- c(1+(x1+x2)*fi[1], fi[1]+x1*fi[2]+x2*fi[3], fi[1]+x2*fi[2]+x1*fi[3], fi[3]+(x1+x2)*fi[4])

# gradient
f11m2 <- function(y,fv,parm){y1 <- y*b1; y2 <- y*b2; 
      list(c(
        b1*fv[2] + b2*fv[3],
        b1*(fv[1]*h2000(y1,y2)+fv[2]*h2010(y1,y2)+fv[3]*h2001(y1,y2)) + b2*fv[4],
       	b2*(fv[1]*h2000(y2,y1)+fv[2]*h2001(y2,y1)+fv[3]*h2010(y2,y1)) + b1*fv[4],
         b1*(fv[1]*h1200(y2,y1)+fv[2]*h1201(y2,y1)+fv[3]*h1210(y2,y1)+fv[4]*h1211(y2,y1))
        +b2*(fv[1]*h1200(y1,y2)+fv[2]*h1210(y1,y2)+fv[3]*h1201(y1,y2)+fv[4]*h1211(y1,y2)))
)}

output <- ode(fi,func=f11m2,(1:totalsteps)*x/totalsteps)
prob0 <- ((b1*b2)^(n/2)*gamma(a)*gamma(a-1/2))/(gamma(c)*gamma(c-1/2)) * x^(n*m/2) * exp(-x*(b1+b2))
cat("x=",x, "prob=", output[totalsteps,2]*prob0,"\n")
\end{verbatim}
\normalsize


\bibliographystyle{plainnat}
\bibliography{hghg}

\end{document}